\documentclass{amsart}

\usepackage{amssymb,amsmath,latexsym,amscd,amsfonts,bbm, enumerate,color,amsthm}  
\usepackage{graphicx}  

\newtheorem{theorem}{Theorem}[section]

\newtheorem{lem}[theorem]{Lemma}

\newtheorem{corol}[theorem]{Corollary}

\theoremstyle{plain}

    \newtheoremstyle{TheoremNum}
        {\topsep}{\topsep}              
        {\itshape}                      
        {}                              
        {\bfseries}                     
        {.}                             
        { }                             
        {\thmname{#1}\thmnote{ \bfseries #3}}
    \theoremstyle{TheoremNum}
    \newtheorem{thmn}{Theorem}

\newcommand{\norm}[1]{{\| #1 \|}}

\newcommand{\C}{\mathbb{C}}
\newcommand{\Z}{\mathbb{Z}}
\newcommand{\ip}[2]{\langle #1 , #2 \rangle}

\begin{document}

\title{Sums of twisted circulants}

\date{July 17, 2016}

\begin{abstract}
The rate of convergence of simple random walk on the Heisenberg group over $\Z/n\Z$ with a standard
generating set was determined by Bump et al \cite{AnExercise,persi_et_al} to be $O(n^2)$.  We extend
this result to random walks on the same groups with an arbitrary minimal symmetric generating set. 
We also determine the rate of convergence of simple random walk on higher-dimensional versions of
the Heisenberg group with a standard generating set.  We obtain our results via Fourier analysis, using
an eigenvalue bound for sums of twisted circulant matrices.  The key tool is a generalization of a version of
the Heisenberg Uncertainty Principle due to Donoho-Stark \cite{DS}.
\end{abstract}

\author[Abrams]{Aaron Abrams}
\address[Aaron Abrams]{Washington and Lee University}
\email{abramsa@wlu.edu}

\author[Landau]{Henry Landau}
\address[Henry Landau]{AT\&T Research}
\email{henry.j.landau@gmail.com}

\author[Landau]{Zeph Landau}
\address[Zeph Landau]{University of California, Berkeley}
\email{zeph.landau@gmail.com}

\author[Pommersheim]{Jamie Pommersheim}
\address[Jamie Pommersheim]{Reed College}
\email{jamie@reed.edu}

\maketitle

\section{Introduction}
In this paper we establish upper bounds on the norms of certain sums of Hermitian matrices.
Our interest in this result is twofold.  First, using standard Fourier analytic techniques developed
by Diaconis \cite{persiBook}, our bounds enable an analysis of rates of convergence for a family 
of random walks on various finite Heisenberg groups, building on work of Bump et al.~\cite{AnExercise}.
Second, we obtain our bound using a type of ``uncertainty principle'' that may be useful in analyzing 
other related problems.

\subsection{Heisenberg groups}\label{sec:intro}

The \emph{Heisenberg group} over a ring $R$ is the set of $3\times 3$ matrices of the form
\[
\left(\begin{array}{ccc}   1&x&z  \\ 0&1&y  \\ 0&0&1 \\ \end{array}\right),
\]
with $x,y,z\in R$, under the usual operation of matrix multiplication.  
In this paper we take $R=\Z/n\Z$ for an odd prime integer $n$.  Letting $(x,y,z)$ denote the above matrix,
the resulting Heisenberg group $H(n)$ is generated by $X=(1,0,0)$ and $Y=(0,1,0)$, 
since $XYX^{-1}Y^{-1}=(0,0,1)$.
This group has order $n^3$ and center $\{(0,0,z)\}$ of size $n$.  

The stipulations that $n$ be odd and prime are mainly for convenience.

The authors of \cite{AnExercise} establish the rate of convergence of random walk on $H(n)$ 
with steps taken uniformly at random from $\{X^{\pm1}, Y^{\pm1}\}$.  Their approach is to
use the standard representation-theoretic techniques of \cite{persiBook}, though there are difficulties.
The group $H(n)$ has $n^2$ one-dimensional representations and $n-1$ irreducible representations 
of dimension $n$.  To apply the standard machinery what is needed are good bounds on the 
eigenvalues of the average of the images of the generators under the $n$-dimensional 
representations of $H(n)$.  This leads to the study (in \cite{AnExercise}) of
the matrices $M(r)$ defined as follows.  Let $S$ be the $n\times n$ ``shift'' matrix which acts on the standard basis by $Se_i=e_{(i-1)\mod n}$ 
and for $r=1,\dots, n-1$, let $D(r)$ be the diagonal matrix whose $j$th entry is $2\cos (2\pi  r j/n)$.
Then for each $r$ the matrix
$$
M(r) = \frac14 (S+S^{-1}+D(r))
$$
is the average of the images of $X^{\pm1}$ and $Y^{\pm1}$ in one of the $n$-dimensional irreducible
representations of $H(n)$.  The papers \cite{AnExercise} and \cite{persi_et_al} present several different
proofs that the norm of $M(r)$ is bounded above by $1- O(\frac{1}{n})$.  Once this is done,
a straightforward analysis of the one-dimensional representations of $H(n)$ reveals that
those eigenvalues can be as large as $1-O(\frac 1{n^2})$, so the bound on $M(r)$ shows
that the behavior of random walk is governed by the one-dimensional representations and a mixing time of 
$O(n^2)$ is established.

In this paper we generalize the first approach utilized in \cite{persi_et_al} in two different ways.
Our main contribution is the analysis of simple
random walks on $H(n)$ with different minimal, symmetric generating sets.
For  $s_1, s_2, r_1, r_2\in \Z/n\Z$, consider random walk on 
$H(n)$ with steps taken uniformly from
 $$G = \{  (s_1, r_1, 0), (-s_1, -r_1, 0), (s_2, r_2, 0), (-s_2, -r_2, 0) \}.$$ 
Note that $G$ generates $H(n)$ if and only if $r_1s_2\not\equiv r_2s_1 \mod n$.
In the case that $G$ generates, the matrices analogous to $M(r)$ that arise in the representation 
theory are averages of matrices that we call {\it twisted circulants}. We will define these matrices shortly.  
The main aim of this paper is to establish the same bound $1-O(\frac 1 n )$ on the eigenvalues of 
these averages of twisted circulants.  The same reasoning then yields mixing times $O(n^2)$ for
random walk on $H(n)$ with generating set $G$. 

Our approach, following \cite{persi_et_al}, is to generalize a version of the Heisenberg Uncertainty
Principle due to Donoho and Stark \cite{DS}.  We discuss this further in Section 
\ref{sec:remarks}.


Another generalization of the main result of \cite{persi_et_al} applies to simple random walk on higher-dimensional 
Heisenberg groups.  The \emph{$d$-dimensional Heisenberg group} consists of upper triangular $(d+2)\times(d+2)$ 
matrices with ones on the diagonal and zeroes everywhere else except the top row and the rightmost column.  With 
entries from $\Z/n\Z$, this is a nilpotent group of order $n^{2d+1}$ and class $d+1$.  

Using a tensor product decomposition of the representations of these groups, we show that the rate of convergence
of simple random walk on these groups with a standard set of generators can be determined easily from the 
lower-dimensional case.  The result is Corollary \ref{cor:higher}: with the standard generators, $dn^2$ steps are 
necessary and sufficient for convergence of simple random walk on the $d$-dimensional Heisenberg group over 
$\Z/n\Z$.  This works easily because the tensor product decomposition is naturally compatible with the standard 
basis.  Probably the same rate occurs with other generating sets, though we do not carry out this analysis.

%
%
%
%

\subsection{Main theorem}\label{sec:results}
We begin with the definition of the twisted circulants.   
We will index the rows and colums of all matrices starting at 0. 
Let $S$ denote the $n\times n$ matrix with $ij$ entry $1$ if $i+1=j \mod n$ and $0$ otherwise.  
Thus $S$ enacts the shift operator on the standard basis of $\C^n$.  A \emph{circulant} is any 
non-negative power of $S$.  
If $D$ is a diagonal matrix with entries $d_i$ and $C=S^s$ is a circulant then the matrix
$DC$ has the entries $d_i$ on the $s$th (cyclic) diagonal above the main: 

\[
\left(\begin{array}{cccccccc} &  &  & d_0 &  &  &  &  \\ &  &  &  & d_1 &  &  &  \\ &  &  &  &  & \ddots &  &  \\ &  &  &  &  &  & \ddots &  \\ &  &  &  &  &  &  &  d_{n-s-1} \\d_{n-s} &  &  &  &  &  &  &  \\ & \ddots &  &  &  &  &  &  \\ &  & d_{n-1} &  &  &  &  & \end{array}\right)
\]

Recall that $n$ is a fixed odd prime. Let $\omega=e^{2\pi i/n}$ and let $R$ be the diagonal matrix with 
$R_{j,j}= \omega^j$.  Note that 
\begin{equation*}\label{eq:almostcommute}
SR=\omega RS.
\end{equation*}
A {\em twisted circulant} is a unitary matrix of the form
\[A(r,s):=R^rS^s,\]
and these have Hermitian counterparts
\[ M(r,s):=\frac 1 2 \left(R^r S^s + (R^rS^s)^*\right)=\frac 1 2 \left( A(r,s)+\omega^{rs}A(-r,-s)\right)\]
with $r,s\in \Z_n$.

The matrices $A(r,s)$ are scalar multiples of the matrices in the images of the $n$-dimensional 
representations of $H(n)$.
Note that the matrix $M(r)$ that arose in \cite{AnExercise} is equal to $\frac 12 (M(r,0)+M(0,1))$.

Our main result (proved in Section \ref{s:4}) is this:

\begin{theorem}\label{thm:main}
If $r_1s_2\not\equiv r_2s_1 \mod n$ then the norm of $\frac 1 2 \left( M(r_1,s_1)+M(r_2,s_2) \right)$ 
is at most $1-O(\frac 1 n)$.
\end{theorem}

If $r_1s_2\equiv r_2s_1 \mod n$ then the operators $M(r_1,s_1)$ and $M(r_2,s_2)$ commute, 
and the analysis is different; we discuss this case in Section \ref{sec:equalslopes}.

\begin{corol} If $r_1s_2\not\equiv r_2s_1 \mod n$ then $n^2$ steps are necessary and sufficient for 
random walk on $H(n)$ with generating set 
$\{  (s_1, r_1, 0), (-s_1, -r_1, 0), (s_2, r_2, 0), (-s_2, -r_2, 0) \}$ to become uniform.
\end{corol}

The next observation can be used in the same way 
to determine mixing times for random walks with larger sets of generators.

\begin{corol} \label{c:1}
Let $r_i,s_i \in \Z_n$ for $i=1,\ldots,d$ and consider the matrix $M=\frac 1 d \sum M(r_i,s_i)$.
If there are $2k$ disjoint pairs of integers in $\{1,2,\ldots,d\}$ such that $r_is_j\not\equiv r_js_i \mod n$
for each pair $i,j$, then the norm of $M$ is at most $1-\frac {2k}{dn}$.  In particular if a constant
fraction of the indices can be paired in this way then the norm of $M$ is at most $1-O(\frac 1 n)$.
\end{corol}

\begin{proof}
The matrix $M$ is the average of $k$ matrices of norm at most $2-\frac 2 n$ and $d-2k$ matrices of norm $1$.
\end{proof}

\subsection{Remarks on uncertainty and Gauss sums}\label{sec:remarks}

Theorem \ref{thm:main} can be framed in the context of the following general question:  given two Hermitian 
matrices $A$ and $B$ what can one say about the norm of the sum $A + B$?  The triangle inequality  
yields an upper bound of $||A||+||B||$ on the largest eigenvalue of $A+ B$, and of course this bound is 
sometimes achieved e.g.~when the leading eigenvectors of $A$ and $B$ coincide.  Horn's inequalities
generalize this observation but due to their generality do not address the ``typical'' situation.
Specifically, suppose the diagonalizing bases are ``well-mixed" with respect to each other, meaning
roughly that the eigenspaces for the leading eigenvalues 
of $A$ and $B$ are very far from aligning.  This is meant to suggest that the two eigenbases 
should satisfy a kind of uncertainty principle:  any vector well-concentrated in the eigenbasis of $A$ should be 
spread out in the eigenbasis of $B$, and vice versa.  Applying this to the leading eigenvector of 
$A + B$ gives a bound on its eigenvalue as a weighted average of the largest parts of the spectra of 
$A$ and $B$.  See Section \ref{sec:up} for precise statements; in the context of twisted circulants, this 
yields Theorem \ref{thm:main}.

This is very similar to the uncertainty principle of Donoho-Stark \cite{DS}.  In their case the eigenspaces for $A$ and $B$ 
are related by the Fourier Transform, but this is not essential.  Their beautiful and simple proof of the uncertainty principle 
relies solely on the fact that the individual matrix entries of the Fourier Transform are of size $O(\frac{1}{\sqrt{n}})$.  In fact 
this is exactly what we mean by ``well-mixed.''  Thus our uncertainty principle, Lemma \ref{l:up}, is just a slight
generalization of the result from \cite{DS}.

For the particular twisted circulants we are interested in, when we computed the change of basis 
matrix (see Section \ref{sec:calc}) we were delightfully surprised by the appearance of Gauss 
sums in our calculation.  These enabled us to establish that the entries have norm exactly $\frac{1}{\sqrt{n}}$, 
and hence the uncertainty principle applies.

\section{Proof of main theorem}
\subsection{Uncertainty Principle} \label{sec:up}
For a set $S \subseteq \{0,\ldots,n-1\}$ and $x \in \C ^n$ define $x_S$ to be the projection of $x$ onto the coordinates of $S$:
$(x_S)_j =x_j$ for $j\in S$ and $(x_S)_j  =0$ for $j$ otherwise.  Let $\bar S$ denote the complement
$\{0,\ldots,n-1\}\setminus S$.

\begin{lem}[Uncertainty principle] \label{l:up}
Suppose $U$ is a unitary matrix and $c$ a positive real number such that 
$|U_{i,j} |\leq \frac{c}{\sqrt{n}}$ for all 
$1 \leq i,j \leq n$.  Then for any sets $S,T\subseteq\{0,\ldots,n-1\}$ and any $v\in \C^n$ of unit norm,
\[ \frac{|S||T|} {n} \geq \left( \frac{\sqrt{1-  \norm{(Uv)_{\bar T}}^2}  - \norm{v_{\bar S} }}{c\norm{v_S}} \right)^2. \]
\end{lem}
\begin{proof} The bound $|U_{i,j} |\leq \frac{c}{\sqrt{n}}$ implies that the maximum absolute value for a coefficient of $Uv_S$ 
is $|S|\frac{c}{\sqrt {n}} \frac{1}{\sqrt{|S|}} \norm {v_S}$ which implies that 
\begin{equation}  \label{e:1} \norm {( Uv_S)_T} \leq  \frac{ c\norm{v_S}}{\sqrt{n}}\sqrt{|S||T|} . \end{equation}
Now $1-  \norm{(Uv)_{\bar T}}^2
= \norm{(Uv)_T}^2  
\leq (\norm {(Uv_S)_T} + \norm{(Uv_{\bar S})_T})^2 
\leq ( \frac{ c\norm{v_S}}{\sqrt{n}}\sqrt{|S||T|} + \norm{v_{\bar S}})^2$ 
where we've used  \eqref{e:1} on the first term and the fact that $U$ is unitary on the second term.  Rearranging terms gives the result.
\end{proof}
 
\begin{corol}  \label{c:up} Let $U$ be a unitary matrix with $|U_{i,j}| \leq \frac{c}{\sqrt{n}}$ for all $1 \leq i,j \leq n$. For any $v\in \C^n$ of unit norm and sets $S,T\subseteq\{0,\ldots,n-1\}$ with $|S| |T| \leq \frac{n}{2c^2}$, we have  $\max (\norm{v_{\bar S}}, \norm{ (Uv)_{\bar T}} ) \geq \frac{1}{5} $.
\end{corol}

\begin{proof}
Since $\norm{v_S} \le 1$, the assumption on the size of $|S||T|$ in Lemma \ref{l:up} yields the 
inequality:
\[ (\sqrt{1-  \norm{(Uv)_{\bar T}}^2}  - \norm{v_{\bar S} })^2 \leq \frac{1}{2}.\]  
It can easily be verified that $\max (\norm{v_{\bar S}}, \norm{ (Uv)_{\bar T}} ) \geq \frac{1}{5} $ is a 
necessary condition for the inequality to hold.
\end{proof}

\subsection{The eigenstructure of DC} \label{sec:calc}

Given an arbitrary matrix $M$, we'll say a unitary $U$ {\em diagonalizes} $M$ if $U^*MU$ is a diagonal matrix.  Note that $U e_i$  ($e_i$ being the $i$th coordinate basis vector) is the eigenvector for $M$ corresponding to the eigenvalue located on the $i$th diagonal element of $U^*MU$.  (Not all matrices have such a unitary but self-adjoint ones and unitaries do.)

 Let $F$ be the $n \times n$ Fourier transform matrix defined by 
$F_{k,l}= \frac 1 {\sqrt{n}} \omega^{kl}$ where $\omega= e^{2\pi i/n}$.
Also, define the permutation matrix $\Pi_s$ by $(\Pi_s)_{si,i} =1$ for $0\leq i \leq n-1$ with the remaining entries $0$;
thus $\Pi_s (x_0\ x_1\ \cdots\ x_{n-1})^*=(x_0\ x_s\ \cdots\ x_{s(n-1)})^*$ (with indices taken mod $n$).
 
\begin{lem}  \label{l:pibf} Let $C=S^s$ be a circulant and $D$ be diagonal with entries $a_i$ of unit norm.  Let
$\alpha=a_0a_1\cdots a_{n-1}$ and let $B$ be the diagonal matrix with entries
\[B_{k,k}= \frac{\lambda_0^k}{\prod_{l=0}^{k-1} a_{sl}}, \]
where $\lambda_0$ is any fixed $n$th root of $\alpha$.
Then the unitary matrix $\Pi_s B F$ diagonalizes $DC$. 
The $(j,j)$ entry of the resulting diagonal matrix $(\Pi_s B F)^* DC (\Pi_s B F)$ is $\omega^{j}\lambda_0$.
\end{lem}

\begin{proof}    Define $\lambda_0$ to be an $n$th root of $\alpha$ so that $\lambda_0, \omega \lambda_0, \omega^2 \lambda_0, \dots \omega^{n-1} \lambda_0$ are all the $n$th roots of $\alpha$.   It is straightforward to verify that 
$\omega^{j}\lambda_0$ is an eigenvalue for $DC$ with eigenvector $x_j = \Pi_s v_j$ where the $k$th
coordinate of $v_j$ is given by
\begin{equation} \label{e:vdef} (v_j)_k=\frac{\omega ^{jk}}{\sqrt{n}} \frac{\lambda_0^k}{\prod_{l=0}^{k-1} a_{sl}} . \end{equation}
 Define $X$ (respectively $V$) to be the matrix whose $j$th column is $x_j$, (respectively $v_j$), so that $X$ is a unitary that diagonalizes $DC$ and $X= \Pi_s V$.   From \eqref{e:vdef}, since the second factor is independent of $j$, we see $V=BF$ with $B$ defined in the statement of the lemma.
\end{proof}

Given $r$, $s$, define $X(r,s)= \Pi_s B F$ where $B$ is a diagonal matrix with $B_{k,k}= \omega ^{-rs\frac{k(k-1)}{2}}$.
\begin{corol} \label{c:3} \label{c:2}
The matrix $X(r,s)$ is unitary, diagonalizes  $R^rS^s$, and has all entries of norm $\frac{1}{\sqrt{n}}$.\end{corol}

\begin{lem} \label{l:2diagonals} Given nonzero elements $r_1,r_2, s_1,s_2$ of $\Z_n$ such that $r_1s_2\ne r_2s_1$,
each entry of  the matrix $(X(r_1,s_1))^* X(r_2,s_2)$ has norm $\frac{1}{\sqrt{n}}$. 
\end{lem}
\begin{proof}
$(X(r_1,s_1))^* X(r_2,s_2) = F^* B_1^* \Pi_{s_1}^ * \Pi_{s_2} B_2^* F= F^* B_1^* \Pi_{s_1^{-1} s_2} B_2^* F$.  
We compute that the $(c,d)$ entry of this product is
\[ [ (X(r_1,s_1))^* X(r_2,s_2) ]_{c,d} = \sum_j \omega ^{-cj} \omega^ {r_1s_1j(j-1)/2}  \omega^{-r_2s_2s_1 s_2^{-1}j(s_1s_2^{-1}j -1)/2} \omega^{s_1 s_2^{-1}jd} .  \]
Let $\Omega$ be the $n$th root of unity satisfying $\omega=\Omega^2$.  (We use here that $n$ is odd.)
Then the above sum becomes 
\[ \sum_j \Omega^{\alpha j^2+\beta j} \]
where $\alpha=r_1s_1-r_2s_2^{-1}s_1^2$ and $\beta=-2c-r_1s_1+r_2s_1+2s_1s_2^{-1}d$.
As long as $\alpha\ne0$, i.e. $r_1s_2\ne r_2s_1$, this is a Gauss sum whose norm is $\sqrt{n}$.
\end{proof}

In the case that $r_1s_2=r_2s_1$, we get $\alpha=0$ and the sum in the proof of the lemma is not 
a Gauss sum.  We discuss this situation in Section \ref{sec:equalslopes}.

\subsection{Putting it together} \label{s:4}

We now prove our main results.  Lemma \ref{l:2diagonals} does not apply in the case $s_1=0$,
but here the analysis is easier and in fact doesn't require that the second operator have entries
coming from $R$.  A special case is $M=\frac 1 2 \left(M(r_1,0)+M(r_1,s_1)\right)$.

\begin{theorem}\label{thm:diagonal}
  Set $M =\frac 1 2 M(r_1,0) + \frac 1 2 \left(D C + (D C)^*\right)$ with $C=S^s$ and 
$D$ any diagonal matrix with entries of unit norm.  Then we have $\norm{M} \leq 1- O(\frac{1}{n})$.
\end{theorem}
\begin{proof}
Let $v$ be a maximal eigenvalue for the self-adjoint operator $M$ so that 
$\norm{M}= \ip{M v}{v} \leq | \ip{M_1 v}{v}| + |\ip{DCv}{v} | + 1 =   | \ip{M_1 v}{v}| + |\ip{D' Uv}{Uv} | + 1$,
where $M_1$ is shorthand for the matrix $M(r_1,0)$ and where the matrix $U$ diagonalizes the unitary matrix 
$DC$ with resulting diagonal matrix $D'$, i.e. $D'= U^*(DC) U$.  By Lemma \ref{l:pibf}, $U= \Pi_s B F$ and 
therefore both $U$ and $U^*$ have all entries of norm $\frac{1}{\sqrt{n}}$ and therefore satisfy the conditions 
for the uncertainty principle (Corollary \ref{c:up}).  Choosing $S=[1, \sqrt{n}/2]$ we have 
$\max (\norm{v_{\bar S}}, \norm{ (U^*v)_{\bar T}} ) \geq \frac{1}{5} $ for any set $T$ with $|T| \leq \sqrt{n}$.  
If this maximum is achieved by the first term, the fact that all the eigenvalues of $D$ on $\bar S$ are bounded by 
$2- O(\frac{1}{n})$ gives the result.  We are left with showing the result under the assumption that 
$\norm{ (U^*v)_{\bar T}} \geq \frac{1}{5}$.  We note that by the proof of Lemma \ref{l:pibf} the eigenvalues of 
$U^*$ are spread evenly around the unit circle.  If we let $\alpha$ be the unit vector in the direction of 
$\ip{D' U^*v}{U^*v}$ and set $f_i= \ip{ (D')_i}{\alpha}$ we have $|\ip{D' U^*v}{U^*v}| = \sum_i f_i |v_i|^2$.  
Choosing $T$ to be the locations $i$ for the top $\sqrt{n}$ values of $f_i$ (in absolute value) and noting 
that the remaining values of $f_i$ must have absolute value below $1- O(\frac{1}{n})$ yields the result.
\end{proof}

%
\begin{thmn}[\ref{thm:main}]
 Let $M = \frac 1 2 \left( M(r_1,s_1)+M(r_2,s_2)\right)$ with $r_1,s_1,r_2,s_2$ nonzero elements 
of $\Z_n$.  If $r_1s_2 \not\equiv r_2 s_1 \mod n$ then $\norm{M} \leq 1- O(\frac{1}{n})$.
\end{thmn}
\begin{proof}  
The argument has lots of similarities to the previous result.  
For shorthand write $A_i=A(r_i,s_i)=R^{r_i}S^{s_i}$ (for $i=1,2$) so that $M(r_i,s_i)=A_i+A_i^*$.  Let $v$ be a maximal eigenvalue for the self-adjoint operator $M$ so that $\norm{M}= \ip{M v}{v} \leq | \ip{A_1 v}{v}| + |\ip{A_2 v}{v} | + 2 =   | \ip{B_1 U_1 v}{U_1v}| + |\ip{B_2 U_2v}{U_2v} | + 2$ where the unitary matrix $U_i=X(r_i,s_i)$ diagonalize $A_i$ with resulting diagonal matrices $B_i$, i.e., $B_i= U_i A_i U_i^*$ for $i=1,2$.
We write $w=U_1v$ and we consider the resulting quantities $|\ip{B_1w}{w}|$ and $|\ip{B_2 (U_2U_1^* w)}{U_2U_1^* w}|$.  Lemma \ref{l:2diagonals} establishes that  $U_2U_1^*$ has all entries bounded by $\frac{1}{\sqrt{n}}$ and therefore satisfies the uncertainty principle.  The remainder of the proof mirrors that of the previous theorem.  
\end{proof}

\section{Equal slopes}\label{sec:equalslopes}

The hypothesis that $r_1s_2\ne r_2s_1$ in $\Z_n$ is required for the proof of Lemma \ref{l:2diagonals}.
If $r_1s_2= r_2s_1$ then the sum in the proof of that lemma is not a Gauss sum, because $\alpha=0$.
For each $c$ there is a single value of $d$ that gives a sum of $1$, and the rest give 
zeroes.  The resulting matrix is a permutation matrix.

From \eqref{eq:almostcommute} it follows that $(A(r,s))^k=\omega^{\frac{k(k-1)}2 rs}A(kr,ks)$.
Thus $A(r,s)$ and $\omega I$ generate an abelian group of matrices including $A(kr,ks)$ for all $k$.
So $r_1s_2= r_2s_1$ implies that $M(r_1,s_1)$ and $M(r_2,s_2)$ commute and generate a group
(usually) isomorphic to $\Z_n^2$.

In this case, the matrix $ M=\frac{1}{2}(M(r_1,s_1) + M(r_2,s_2))$ will have norm close to 1 for some choices of the parameters, and will not have norm close to 1 for other choices of parameters.  Specifically, we note that the eigenvalues of $M$ are given as follows.  Let $k=r_1^{-1}r_2 = s_1^{-1} s_2$, with the inverses taken modulo $n$.  Then the eigenvalues of $M$ are given by:
$$
\lambda_d = \frac{1}{2}\left(\cos \frac{2\pi d}{n} + \cos \frac{2\pi (\frac{-k(k-1)}{2}r_1s_1+kd)}{n}\right), \ \ \ d=0,\dots, p-1.
$$
We see that for a given $n$,  these eigenvalues only depend on $k$ and the product $r_2s_2$. For $n=401$, Figure
\ref{fig} shows the values of these parameters for which the matrix $M$ has norm greater than $1-\cos(\frac{2\pi}{n})$.  In the graph, the horizontal axis gives the value of $r_2s_2$ while the vertical axis gives the value of $kr_2s_2$.

\begin{figure}[htbp]
\begin{center}
\includegraphics[height=6in]{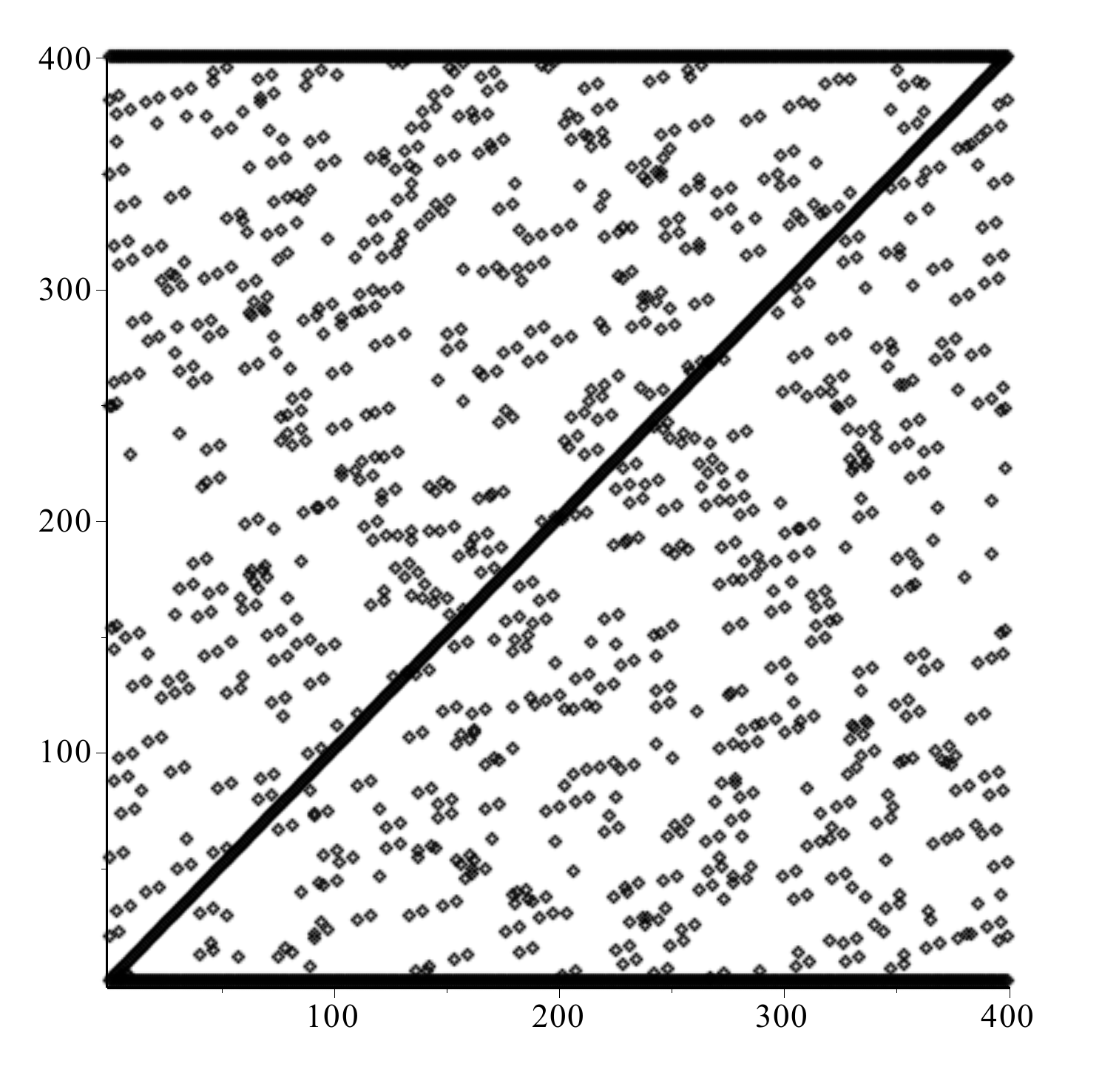}
\caption{A mark indicates that $M$ has eigenvalues greater than $1-\frac 1 n$ for the given choices of 
$r_1s_1$ and $k$.  Here $n=401$.} 
\label{fig}
\end{center}
\end{figure}

On the other hand, one finds computationally, again for $n=401$, that the norm of $M$ is less than $1-\frac{1}{n}$ for approximately half of the choices of the parameters.

\section{Higher dimensional Heisenberg groups}\label{sec:higher}

An analysis very similar to what we have already done can be carried out for the higher dimensional
Heisenberg groups defined in Section \ref{sec:results}.  We denote this group by $H(p,d)$ where $p=n$ is
still an odd prime.  As $p$ and $d$ will not change we also refer to this simply as $H$.
It is a $p$-group of order $p^{2d+1}$ and nilpotency class $d+1$.

Here we analyze random walk on $H$.
We briefly describe the representation theory of $H$, as communicated to us by Persi Diaconis. 
There are $p^{2d}$ one-dimensional irreducible representations and $p-1$ irreducible representations 
of dimension $p^d$.  These latter representations are described as follows.   We view elements of $H$ as triples $(x,y,z)$, where $x,y\in\Z_p^d$, and $z\in\Z_p$.  Let $q=\exp \frac{2\pi i }{p}$ and let $V$ be the vector space of all complex-valued functions on $(\Z_p)^d$.   Then for each $0\neq c\in \Z_p$, there is an irreducible representation 
$\rho_c$ of $H$ on $V$ given by
$$
[\rho_c(x,y,z)f](w) = q^{c(y\cdot w+z)} f(w+x).
$$

The key to understanding these representations, and the random walk, is a tensor product decomposition.
Let $W$ denote the vector space of complex-valued functions on $\Z_p$, and note that the $V$ is naturally isomorphic to $W^{\otimes d}$.  Define operators $S$ and $R$ on $W$ by
$$
[Sg](u)=g(u+1),\ \ \ [Rg](u) = q^u g(u),
$$
where $S$ and $R$ are the matrices defined in Section \ref{sec:results}.
Then one easily verifies that the operator $\rho_c(x,y,z)$ decomposes as
$$
\rho_c(x,y,z) = q^{cz} \biggl[ R^{cy_1} S^{x_1} \otimes \cdots \otimes R^{cy_d} S^{x_d}\biggr]. 
$$

We are interested in bounding the top eigenvalue of the average of $\rho_c$ on the size $4d$ generating set
consisting of the $e_i=(w_i,0,0)$ and the $f_i=(0,w_i,0)$ and their negatives, where $w_i$ denotes the $i$th 
vector in the standard basis of $\Z_p^d$.
We see that
$$
\rho_c(e_i) = I \otimes \cdots I \otimes S\otimes I \cdots \otimes I,
$$
with the $S$ appearing in the $i$th tensor factor, and 
$$
\rho_c(f_i) = I \otimes \cdots I \otimes R^c\otimes I \cdots \otimes I.
$$

So for fixed $i$ the average $\frac 1 4 (\rho_c(e_i) + \rho_c(-e_i) +  \rho_c(f_i) + \rho_c(-f_i))$
is precisely the sum $\frac 1 2 M(0,1)+ \frac 1 2 M(c,0)$, tensored with a bunch of identity matrices.
Recall from Section \ref{sec:intro} that $\frac 1 2 M(0,1)+ \frac 1 2 M(c,0)$ is the same as the matrix 
$M(c)$ studied in \cite{AnExercise}.
Thus, by the results of \cite{AnExercise} or \cite{persi_et_al}, or by Theorem \ref{thm:diagonal},
the top eigenvalue of such an operator is at most $1-O(1/p)$.
Averaging these over $i$ again produces an operator with top eigenvalue at most $1-O(1/p)$.

By contrast, the one-dimensional representations of $H=H(p,d)$ map each of the $4d$ generators $\{\pm e_i, \pm f_i\}$
to a $p$th root of unity, so the largest value for the average of these $4d$ numbers arises when all but one have value
$1$ and the other is $\exp(2\pi i/p)$.  This comes out to something larger than $1-O(\frac 1 {dp^2})$, which we have
shown is much larger than the contribution from the high-dimensional representations.  Thus again in this case the rate of
convergence is governed by the one-dimensional representations.


\begin{corol}\label{cor:higher}
For simple random walk on $H(p,d)$ with steps $\{\pm w_i,0,0),(0,\pm w_i,0)\}$ each chosen with the same 
probability $1/4d$, the mixing time is $O(dp^2)$.
\end{corol}

\end{document}